\documentclass[11pt,regno]{amsart}
\usepackage{amsmath,amsthm,mathrsfs,amsfonts,amssymb,amscd,euscript,graphicx,overpic,color}

\newtheorem{theorem}{Theorem}
\newtheorem*{theorem*}{Theorem}
\newtheorem{lemma}{Lemma}

\theoremstyle{definition}

\newtheorem*{remark*}{Remark}

\newcommand{\eqdef}{\stackrel{\scriptscriptstyle\rm def}{=}}

\DeclareMathOperator{\diam}{diam}

\DeclareMathOperator{\Jac}{Jac}

\DeclareMathOperator{\supp}{supp}
\DeclareMathOperator{\card}{card}

\def\bZ{\mathbb{Z}}
\def\bR{\mathbb{R}}

\def\cC{\EuScript{C}}

\def\cP{\EuScript{P}}
\def\cR{\mathscr{R}}

\def\cW{\EuScript{W}}
\def\cM{\EuScript{M}}

\DeclareMathSymbol{\varnothing}{\mathord}{AMSb}{"3F}

\author{} \address{}
\email{}
\urladdr{}

\author{Katrin Gelfert} \address{Instituto de Matem\'a‡tica, Universidade Federal do Rio de Janeiro, Cidade Universit\'aria -- Ilha do Fund\~ao, 21945-909 Rio de Janeiro, RJ, Brazil}
\email{gelfert@im.ufrj.br}

\begin{document}

\title[Horseshoes]{Horseshoes  for diffeomorphisms\\ preserving hyperbolic measures}

\begin{abstract}
	We give extensions of Katok's horseshoe constructions, comment on related results, and provide a self-contained proof. We consider either a $C^{1+\alpha}$ diffeomorphism preserving a hyperbolic measure or a $C^1$ diffeomorphism preserving a hyperbolic measure whose support admits a dominated splitting.
\end{abstract}

\thanks{KG was supported by CNPq (Brazil). She is grateful for discussions with Keith Burns, Yang Jiagang, Rafael Potrie, and {Micha\l} Rams. She also thanks the referee for correcting the former proof of Theorem~\ref{t:2} suggesting reference~\cite{Wal:81}.}

\keywords{Pesin theory, non-uniformly hyperbolic dynamics for $C^{1+\alpha}$ and $C^1$ diffeomorphisms, horseshoes, pressure}
\subjclass[2000]{ %
37E05, 
37D25, 
37D35, 
37C45, 
28D99
}
\maketitle
\section{Introduction}

Given a diffeomorphism $f$ of a smooth Riemannian manifold $M$, we call a set $\Gamma\subset M$  \emph{basic} (with respect to $f$) if it is compact and with respect to $f$ invariant,  locally maximal, and hyperbolic.  Let  $\cM(f|_\Gamma)$ denote the space of $f$-invariant Borel probability measure on $\Gamma$. We call an ergodic $f$-invariant Borel probability  measure \emph{hyperbolic} if it has positive and negative but no zero Lyapunov exponents. In this paper we consider in parallel two assumptions:
 \begin{itemize}
	\item[(A1)]  $f$ is $C^{1+\alpha}$ diffeomorphism preserving a hyperbolic ergodic measure $\mu$ with positive entropy, or that\\[-0.3cm]
	\item[(A2)] $f$ is a $C^1$ diffeomorphism preserving a hyperbolic ergodic measure $\mu$ with positive entropy  whose support admits a dominated splitting.\\[-0.3cm]
\end{itemize}

\begin{theorem}\label{t:1}
	 Let $f$ be a $C^1$ diffeomorphism of a smooth Riemannian manifold $M$ and $\mu$ a hyperbolic ergodic $f$-invariant Borel probability measure with positive entropy $h_\mu(f)>0$. 
	 Denote $\chi(\mu)=\min_j\lvert\chi_j(\mu)\rvert$, where $\chi_j(\mu)$ denote the Lyapunov exponents of $\mu$. 	 
	 Assume  either hypothesis (A1) or hypothesis (A2).
	 
	 Then for every $e\in(0,h_\mu(f)]$ and every sufficiently small $r>0$ there exists a set $\Gamma\subset M$ and a positive integer $m$ so that we have
\begin{itemize}
		\item[i)] $\Gamma$  is basic and topologically mixing (with respect to $f^m$),	
		\item[ii)] $\Gamma$ is contained in the $r$-neighborhood of the  support of $\mu$.
		\item[iii)] $\displaystyle d(\nu,\mu)<r$ for every $\nu\in\cM(\Gamma)$, where $d$ is a metric that generates the weak$\ast$ topology,
		\item[iv)] $\displaystyle 
			\big\lvert h(f|_\Gamma) - e\big\rvert<r$,
		\item[v)] For every continuous function $\varphi\colon M\to\bR$ the topological pressure of $\varphi$ (with respect to $f|_{\widehat\Gamma}$, where $\widehat\Gamma=\Gamma\cup f(\Gamma)\cup\ldots\cup f^{m-1}(\Gamma)$) satisfies 
			\[
			\Big\lvert P(\varphi,f|_{\widehat\Gamma})
			- \big( e +\int\varphi\,d\mu \big)
				\Big\rvert <r\,,
			\]	
		\item[vi)] for every Lyapunov regular point $x\in\Gamma$  we have
			\[
				\Big\lvert \frac1m\log\,\lvert \Jac df^m_{|E_x^u}\rvert -
				 	\int \log\,\lvert  \Jac df_{|E^u}\rvert \,d\mu\Big\rvert<r,
			\]	
	\end{itemize}
\end{theorem}

Results of this type are widely referred to Katok~\cite{Kat:80} or Katok and Mendoza~\cite[Supplement S.5]{KatHas:95} (see also~\cite{BarPes:06,BarPes:07}). 

Let us mention some extensions to other settings. An earlier related statement for continuous and piecewise monotone maps  of the interval goes back to Misiurewicz and Szlenk~\cite{MisSzl:80}. 
Corresponding properties of Theorem~\ref{t:1} in the case of a measure with only positive Lyapunov exponent of holomorphic maps were shown by Przytycki and Urba\'nski~\cite{PrzUrb:10}. The case of general $C^{1+\alpha}$ maps is discussed by Chung~\cite{Chu:99} and Yang~\cite{Yan:} (see also Buzzi~\cite{Buz:} for a sketch) and of maps with singular or critical points by Gelfert~\cite{Gel:10}. A related result  is established in the setting of dyadic diophantine approximations by Persson and Schmeling~\cite{Per:06, PerSch:08}. Katok~\cite[Section 4]{Kat:82} gave a version for flows.

Mendoza in his PhD thesis~\cite{Men:83} showed that in the case of a $C^2$ surface diffeomorphism leaving invariant an ergodic hyperbolic SRB measure, there exists a sequence of horseshoes with unstable dimension (the Hausdorff dimension of its intersection with local unstable manifolds) converging to 1 (see~\cite{Men:88}). See also~\cite{Men:85,Men:89} for related work. However, Mendoza's  general arguments that the Hausdorff measure of the approximating horseshoes is uniformly bounded away from zero are not correct~\cite{San:09}.  
S\'anchez-Salas in his PhD thesis~\cite{San:99} continued studying a $C^2$ diffeomorphism of a compact manifold of dimension $\ge2$ leaving invariant an ergodic hyperbolic SRB measure and showed that this measure can be approximated by ergodic measures supported on horseshoes with arbitrarily large unstable dimension~\cite{San:02}. His studies are based on Markov towers that can also be described by horseshoes with infinitely many branches and variable return times, this enables uniform bounds on the distortion. In~\cite{San:03} he also establishes a version for interval maps using Markov towers with good fractal properties. In~\cite{LuzSan:11} Luzzatto and S\'anchez-Salas study ``variable-time horseshoes'' in order to achieve weak$\ast$ convergence of all measures on the horseshoes towards the hyperbolic one (compare property iii) in the above theorem).

In the $C^1$ setting Pesin-Katok theory does not work in general (e.g.~\cite{BonCroKat:13} provide examples of $C^1$ diffeomorphisms and hyperbolic measures without stable/unstable manifolds). Gan~\cite{Gan:02} discusses extensions to $C^1$ surface diffeomorphisms. He restricts to the case that the set of all preperiodic points of hyperbolic (neither sink nor source) periodic points possesses a dominated splitting. This allows him to apply a result by Pujals and Sambarino~\cite{PujSam:00} stating that corresponding local un-/stable manifolds have uniform size and applying a closing lemma by Liao~\cite{Lia:79}. Detailed construction of invariant manifolds for dominated hyperbolic measures are due to Abdenur, Bonatti, and Crovisier~\cite[Section 8]{AbdBonCro:11}. In the proof of Theorem~\ref{t:1} in the $C^1$ setting we will make use of the method of fake foliations presented in~\cite{BurWil:10}. 

\smallskip
The following result is an almost immediate consequence of Theorem~\ref{t:1}.

\begin{theorem}\label{t:2}
	If $h(f|_M)>0$ then for every $e\in[0,h(f|_M))$ there exists a compact $f$-invariant  and  hyperbolic set $\Gamma\subset M$ such that \[h(f|_\Gamma)=e\,.\]
\end{theorem}

Results related to Theorem~\ref{t:2} were also studied by Urba\'nski~\cite{Urb:86,Urb:87} studying the Hausdorff dimension of invariant sets of expanding maps on a circle with a hole and showing that the dimension varies continuously with the size of the hole. See also Persson~\cite{Per:06} and Sun~\cite{Sun:10} for similar results.

\smallskip

In Section~\ref{s:2} we collect preliminary results on Pesin theory (for the $C^{1+\alpha}$ setting) and on dominated splitting (for the $C^1$ setting). Theorems~\ref{t:1} and~\ref{t:2} are proved in Section~\ref{s:3}.

\section{Preliminaries}\label{s:2}

\subsection{Lyapunov regularity and hyperbolicity}\label{s:2.1}
We say that a point $x\in M$ is \emph{Lyapunov regular} if there exist a number 
$s(x) \ge1$,  
numbers $\chi_1(x)<\ldots<\chi_{s(x)}(x)$, and a decomposition $ E^1_x\oplus\cdots\oplus E^{s(x)}_x=T_xM$ into subspaces of dimension $k_j(x)$ such that for $j=1, \ldots, s(x)$  we have
\[
 \chi_j(x)=\lim_{|k|\to\infty}\frac{1}{k}\log\,\lVert df^k_x(v)\rVert
\]
whenever $v\in E^j_x\setminus\{0\}$. It is a consequence of the Oseledets multiplicative ergodic theorem that with respect to any $f$-invariant Borel probability measure on $M$ the set of Lyapunov regular points has full measure. Moreover, the function $x\mapsto \chi_j(x)$ and the distribution $x\mapsto E^j_x$ are Borel measurable and $f$-invariant. A Lyapunov regular point is said to be \emph{hyperbolic} if $s(x)\ge2$ and there exists $i(x)\ge1$ so that $\chi_{i(x)}(x)<0<\chi_{i(x)+1}(x)$. The set of points having this property will be denoted by $\cR_H$, that is,
\[
\cR_H\eqdef\left\{x\in M\text{ Lyapunov regular and hyperbolic}\right\}.
\] 
In this case we also use the notations $E^s_x=E^1_x\oplus\cdots\oplus E^{i(x)}_x$, $E^u_x=E^{i(x)+1}_x\oplus\cdots\oplus E^{s(x)}_x$.

Let  $\cM$ denote the space of $f$-invariant Borel probability measure on $M$  and let $\cM_{\rm erg}$ be the subspace of ergodic measures.
For $\mu\in\cM_{\rm erg}$, $\chi_j$ and $s$ are constant $\mu$-a.e. and we denote these constants by $\chi_j(\mu)$ and $s(\mu)$.
An ergodic Borel probability measure $\mu$ is \emph{hyperbolic} if the set of hyperbolic Lyapunov regular points has full measure. In this case we also use the notation
\[
\chi=\chi(\mu)\eqdef\min_j\,\lvert\chi_j(\mu)\rvert.
\]

Henceforth we consider a general diffeomorphism $f$ preserving a hyperbolic ergodic measure $\mu$. 

\subsection{Uniform geometric potential}

Consider the function $\varphi^u\colon \cR_H\to\bR$ 
\[
	\varphi^u(x)\eqdef -\log \,\lvert \Jac df_{|E^u_x}\rvert 
\]
that is also called \emph{geometric potential}.
In the case that $E^u_x$ is  one-dimensional, 
then we simply have $\varphi^u(x)= -\log \,\lVert df_{|E^u_x}\rVert$. Recall that
\[
	\int \varphi^u\,d\mu = - \int\log\,\lvert \Jac df_{|E^u}\rvert \,d\mu
	= -\sum_j\chi_j(\mu)^+\,,
\]
where $a^+=\max\{0,a\}$.
By the Oseledets multiplicative ergodic theorem, $\varphi^u$ is an measurable function and for $\mu$-almost every $x$ we have 
\[
	\sum_j\chi_j(x)^+
	=\sum_j\chi_j(\mu)^+
	=\lim_{n\to\infty}\psi_n(x)\,,
\]
where 
\begin{equation}\label{def:psiss}
	\psi_n(x)\eqdef \log\,\lvert \Jac df^n_{|E^u_x}\rvert^{1/n}
\end{equation}
is a sequence of measurable functions  $\psi_n\colon\cR_H\to\bR$ which converges $\mu$-almost everywhere to $\sum_j\chi_j(\mu)^+$ as $n\to\infty$. Hence, the following lemma is an immediate consequence of the multiplicative ergodic theorem and the Egorov theorem.

\begin{lemma}\label{lem:mulergthe}
	Given  $\delta\in(0,1)$ and $r>0$, there exist a compact set $\Gamma_J=\Gamma_J(\delta,r)$ and a positive integer $n_J=n_J(\delta,r)$ such that $\mu(\Gamma_J)>1-\delta$ and for every $x\in \Gamma_J$ and every $n\ge n_J$ we have
\[
	\Big\lvert 
	\frac1n \log\,\lvert \Jac df^n_{|E^u_x}\rvert
	 - \sum_j\chi_j(\mu)^+ 
	 \Big\rvert \le r\,.
\]
\end{lemma}

\subsection{Oseledets-Pesin $\varepsilon$-reduction}\label{s:2.400red}

The following results are consequences of the Oseledets-Pesin $\varepsilon$-reduction theorem (see~\cite[S.2 and Theorem S.2.10]{KatHas:95}). For any sufficiently small $\gamma>0$  one can chose  Borel functions $C_1$, $C_2\colon \cR_H\to(0,\infty)$ such that the following holds.
Given $x\in\cR_H$, the subspaces $E^s_x$ and $E^u_x$ 
satisfy
\begin{gather}
\tag{H$1_s$}\lVert df ^n_x(v)\rVert
\le C_1(x) e^{-n(\chi-\gamma)}\lVert v\rVert \quad
\text{ for }v\in E^s_x \text{ and }n>0,\\
\tag{H$1_u$}\lVert df ^{-n}_x(w)\rVert
\le C_1(x) e^{-n(\chi+\gamma)}\lVert w\rVert\quad
\text{ for }w\in E^u_x \text{ and }n>0.
\end{gather}
The angle between these subspaces satisfies
\[\tag{H2}\angle(E^s_x,E^u_x)\ge C_2(x).\] 
Furthermore the functions $C_1$, $C_2$ can be chosen to be \emph{$\gamma$-tempered}, that is, for any $n\in\bZ$ they satisfy
\[\tag{H3}
C_1(f ^n(x))\le C_1(x)\,e^{\gamma\lvert n\rvert}, \quad
C_2(f ^n(x))\le C_2(x)\,e^{\gamma\lvert n\rvert} .
\]

Given $\ell\ge 1$, we consider the \emph{Lyapunov regular set} \emph{of level $\ell$} defined by
\begin{equation}\label{constant}
\cR_H^\ell \eqdef\left\{ x\in\cR_H\colon C_1(x)\le\ell, C_2(x)\ge \ell^{-1}\right\}.
\end{equation}
For  $\mu\in\cM_{\rm erg}$ being hyperbolic, we have  $\mu(\cR_H^\ell)\to 1$ as  $\ell\to\infty$.
Observe that $x\mapsto E^a_x$ is continuous on each $\cR_H^\ell$, $a=s,u$, and that $f$ restricted to ${\cR_H^\ell}$ is  hyperbolic. However, $\cR_H^\ell$ in general fails to be $f$-invariant. The set $\bigcup_{\ell\ge 1}\cR_H^\ell$ is $f$-invariant but in general fails to be  hyperbolic.

It is convenient to introduce a new scalar product structure on the tangent bundle $T_{\cR_H} M$. Given $0<\gamma<\chi/3$, the  so-called \emph{Lyapunov inner product} is defined by
\[
\langle v,w\rangle'_x\eqdef
\sum_{n=0}^\infty\langle df^n_x(v),df^n_x(w)\rangle
			e^{-2n\chi_{i(x)}(x)}e^{ -2 n \gamma}
\]
if  $v$, $w\in E^s_x$, and
\[
\langle v,w\rangle'_x\eqdef
\sum_{n=0}^\infty\langle df^{-n}_x(v),df^{-n}_x(w)\rangle
			e^{2n\chi_{i(x)+1}(x)}e^{-2 n \gamma}
\]
if  $v$, $w\in E^u_x$. Based on the properties (H1)-(H3) one can indeed verify the convergence of the above series. Then $\langle\cdot,\cdot\rangle'_x$ is extended to all vectors by declaring the subspaces $E^s_x$ and $E^u_x$ to be mutually orthogonal with respect to  $\langle\cdot,\cdot\rangle'_x$, that is, for vectors $v=v_1+v_2$ and $w=w_1+w_2$ with $v_1$, $w_1\in E^s_x$ and $v_2$, $w_2\in E^u_x$ it is defined by 
\[\langle v,w\rangle'_x\eqdef  \langle v_1,w_1\rangle'_x+ \langle v_2,w_2\rangle'_x.\] This scalar product induces the so-called \emph{Lyapunov norm} $\lVert\cdot\rVert_x'$ on $T_xM$.

For each Lyapunov regular point $x\in \cR_H$ one can show that there exists a Lyapunov change of coordinates, that is, a linear transformation $C_\gamma(x)\colon\bR^{\dim M}\to T_xM$, such that 
\[
	\langle v,w\rangle=\langle C_\gamma(x) v,C_\gamma(x)w\rangle'_x
\]
and that the matrix
\[
	A_\gamma(x)=C_\gamma(f(x))^{-1}\circ df_x\circ C_\gamma(x)
\]
 has Lyapunov block form
\begin{equation}\label{eq:blockk}
 	A_\gamma(x)=\left(\begin{matrix}
				A_\gamma^1(x)&&\\
				&\ddots&\\
				&&A_\gamma^{s(x)}(x)\\
				\end{matrix}\right)\,,
 \end{equation}
 where each block $A_\gamma^j(x)$ is a $k_j(x)\times k_j(x)$-matrix satisfying
\[
 	e^{\chi_j(x)-\gamma}
	\le \lVert A_\gamma^j(x)^{-1}\rVert^{-1}
	\le\lVert A_\gamma^j(x)\rVert
	\le e^{\chi_j(x)+\gamma}\,.
\] 
Moreover, the sequence $(C_\gamma(f^n(x)))_{n\in\bZ}$ is a $\gamma$-tempered sequence of linear transformations, that is, for any $n\in\bZ$ we have
\[
	\lVert C_\gamma(f^n(x))\rVert
	\le \lVert C_\gamma(x)\rVert e^{\gamma\lvert n\rvert}
	\quad\text{ and }\quad
	\lVert C_\gamma(f^n(x))^{-1}\rVert
	\le \lVert C_\gamma(x)^{-1}\rVert e^{\gamma\lvert n\rvert}.
\]

For $x\in\cR_H$ and $\rho>0$ let 
\[
	T_xM(\rho)\eqdef\{v\in T_xM\colon \lVert v\rVert\le \rho\}.
\]	 
Let us now choose $\rho_0>0$ so that for each $x\in M$ the exponential map $\exp_x\colon T_xM(2\rho_0)\to M$ is an embedding satisfying $\lVert(d\exp_x)_v\rVert\le2$ and so that $\exp_{f(x)}$ is injective on $\exp_{f(x)}^{-1}\circ f\circ \exp_x(T_xM(2\rho_0))$. Define the Lyapunov chart
\[
	\psi_x \eqdef \exp_x\circ \,C_\gamma(x)\colon\bR^{\dim M}\to M
\] 
and define the ``lifted map'' $F_x\colon \bR^{\dim M}\to \bR^{\dim M}$ by
\begin{equation}\label{def:Fx}
	F_x\eqdef \psi_{f(x)}^{-1}\circ f\circ \psi_x.
\end{equation}

By setting 
\[
	q(x)\eqdef r_0\min\{\lVert C_\gamma(x)\rVert,\lVert C_\gamma(x)^{-1}\rVert\}
\]	 
we obtain a $\gamma$-tempered function $q$
 and embeddings $\psi_x\colon B(0,q(x))\to M$ such that $\psi_x(0)=x$. 
Moreover, the map $F_x$ defined above is well-defined on $B(0,q(x))$ and its  linearization $(dF_x)_0$ 
has Lyapunov block form.

Observe that the above facts so far used only the fact that $f$ is differentiable and that the  cocycle $A_\gamma\colon \cR_H\to GL(\dim M,\bR)$ satisfies $\log^+\lVert A_\gamma^{\pm1}\rVert  \in L^1(\mu)$ for every invariant hyperbolic probability measure $\mu$.

\subsection{Rectangular neighborhoods of regular points}\label{s:2.4}

We will make use of a suitable collection of sets that are ``almost rectangles''. 
Let 
\[
	L_x(w)\eqdef \sqrt{\dim M}^{-1}\,q(x)\, w
\]
be the linear rescaling $L_x\colon [-1,1]^{\dim M}\to[-q(x),q(x)]^{\dim M}$ of the unit cube onto the maximal cube contained in $B(0,q(x))$.
With the above introduced parametrization $\psi_x$ in the Lyapunov chart, given $h\in(0,1]$ we define
\[
R(x,h)\eqdef (\psi_x\circ L_x)([-h,h]^{\dim M})\subset M
\]
the \emph{$h$-scaled rectangle centered at} $x\in\cR_H$. 
For each regular point $x\in \cR_H$ the set $R(x,1)$ is also called  \emph{regular neighborhood of  $x$} or \emph{Lyapunov chart at $x$}.

The so defined regular neighborhoods of Lyapunov regular points admit local coordinates with respect to which the dynamics behaves uniformly hyperbolic. Recall that vectors tangent to the stable direction $E^s_x$ (the unstable direction $E^u_x$) are contracted forward in time (backward in time). Fix some constant $L\in(0,1/2)$. Considering the splitting of the tangent bundle $T_xM=E^s_x\oplus E^u_x$ and the corresponding splitting of the unit cube $[-1,1]^{\dim M}=I^s\oplus I^u$, an \emph{admissible stable $(s,L,h)$-manifold}  is the set
\[
	\big\{(\psi_x\circ L_x)(v,\xi(v))\colon v\in I^s\big\}\,,
\]
where $\xi\colon I^s\to I^u$ is a smooth map with Lipschitz constant $L$. An \emph{admissible unstable $(u,L,h)$-manifold} is defined analogously. Such stable/unstable admissible manifolds endow $R(x,h)$ with a product structure as they intersect transversally in a unique point with an angle bounded from zero (whose bound, however, depends on $x$).
This enables us to define the concept of an \emph{admissible stable rectangle} (\emph{admissible unstable rectangle}). An \emph{admissible $L^s$-rectangle} is a set in $R(x,h)$ whose boundaries are smooth sets foliated by segments of admissible stable $(s,L,h)$-manifolds and unstable admissible unstable $(u,L,h)$-manifolds such that the stable manifolds stretch fully across $R(x,h)$. An \emph{admissible $L^u$-rectangle} in $R(x,h)$ is defined analogously. 

In this section we will sketch the proof of the following lemma.

\begin{lemma}\label{lllemma}
	Given numbers $\delta\in(0,1),\gamma,r\in(0,\chi/3)$, and $\varepsilon>0$, there are a compact set $\Gamma_H=\Gamma_H(\gamma,\varepsilon,\delta)$, a positive number $h$, a finite collection of closed rectangles $\{R(x_i)=R(x_i,h)\}_{i=1}^\ell$ 	in  neighborhoods of certain points $x_i\in \Gamma_H$ and positive numbers $\lambda\in(0,1),\rho,L$ satisfying the following: 
\begin{itemize}
	\item[\emph{0.}] $\mu(\Gamma_H)>1-\delta$ and 
				$e^{-\chi-\gamma}\le\lambda\le e^{-\chi+\gamma}$,\\[-0.2cm]
	\item[\emph{1.}] $\Gamma_H\subset\bigcup_{i=1}^\ell B(x_i,\rho)$, and $B(x_i,\rho)\subset{\rm int}\, R(x_i)$ for each $i$,\\[-0.2cm]
	\item[\emph{2.}]  $\displaystyle{\rm diam}\,R(x_i)<\varepsilon$ for each $i$,\\[-0.2cm]
	\item[\emph{3.}] if $x\in \Gamma_H\cap B(x_i,\rho)$ and $f^m(x)\in \Gamma_H\cap B(x_j,\rho)$ for some $m\ge 1$, then the connected component of $R(x_i)\cap f^{-m}R(x_j)$ which contains $x$, 
	\[
		\cC_x\left(R(x_i)\cap f^{-m}R(x_j)\right),
	\] 
	 is an admissible $L^s$-rectangle in $R(x_i)$ and its image 
	\[
		f^m\left(\cC_x\left(R(x_i)\cap f^{-m}R(x_j)\right)\right)
	\]	 
	is an admissible $L^u$-rectangle in $R(x_j)$,\\[-0.2cm]
	\item[\emph{4.}] if $x$, $f^m(x)\in \Gamma_H\cap B(x_i,\rho)$ for some $m\ge 1$ then for every $k=0,\ldots, m$ we have
	\[
		\diam f^k\left(\cC_x\left(R(x_i)\cap f^{-m}R(x_i)\right)\right)
		\le  3\diam R(x_i)\max\{\lambda^k,\lambda^{m-k}\},
	\]	
	\item[\emph{5.}] if $x$, $f^m(x)\in \Gamma_H\cap B(x_i,\rho)$ for some $m\ge 1$ then
	\[
		\Big\lvert \frac1m \log\,\lvert \Jac df^m_{|E_x^u}\rvert-
			      \frac1m \log\,\lvert \Jac df^m_{|E_{x_i}^u}\rvert\Big\rvert<r\,.
	\]
\end{itemize}
\end{lemma}

We will split the sketch of the proof of the above lemma in two parts,   assuming first (A1) and then (A2).

\begin{proof}[Sketch of proof of Lemma~\ref{lllemma} assuming (A1)]
Let $f$ be $C^{1+\alpha}$ for some $\alpha>0$. 

Let us write the map $F_x$ defined in~\eqref{def:Fx} as
\[
	F_x(w)=D_\gamma(x)(w)+h_x(w),\quad\text{ where }
	D_\gamma(x)\eqdef C_\gamma(f(x))^{-1}\circ df_x\circ C_\gamma(x).
\]
Note that $(dF_x)_0=D_\gamma(x)$ and $(dh_x)_0=0$.  Recall that the linearization $(dF_x)_0=A_\gamma(x)$ of the map $F_x$ in $0$ has Lyapunov block form~\eqref{eq:blockk}.
The map $F_x$ and its linearization $(dF_x)_0$ are $\gamma$-close in the $C^1$ topology in $B(0,q(x))$.  
Further, there exists a constant $K>0$ and a measurable $\gamma$-tempered function $A\colon\Gamma'\to\bR$ on a set $\Gamma'\subset \cR_H$ of full measure such that  for every $y$, $z\in B(0,q(x))$ we have
\[
	K\,d(\psi_x(y),\psi_x(z))
	\le \lVert y-z\rVert
	\le A(x)\,d(\psi_x(y),\psi_x(z))\,.
\] 
Further, there is a measurable $\gamma$-tempered function $B\colon\Gamma'\to\bR$ such that for every $y, z\in B(0,q(x))$ we have
\[
	B(x)^{-1}d(\psi_x(y),\psi_x(z))
	\le d'_x(\exp_x(y),\exp_x(z))
	\le B(x)\,d(\psi_x(y),\psi_x(z)),
\]
where $d'_x$ denotes the distance on $\exp_xB(0,q(x))$ with respect to the Lyapunov metric $\lVert\cdot\rVert_x'$.
For details on and proofs of the above sketched properties see~\cite[Chapters S.3 and S.4 b--d]{KatHas:95}  and~\cite{BarPes:06, BarPes:07} and, in particular,~\cite[Theorems S.3.1 and S.4.16]{KatHas:95}. 

There exists a positive integer $\ell$ such that $\mu(\cR_H^\ell)\ge 1-\delta/2$. Hence, the set $\Gamma_H=\Gamma'\cap \cR_H^\ell$ has measure at least $1-\delta$ and on $\Gamma_H$ the above considered functions and the size of the Lyapunov charts are bounded.
Moreover, by the Lusin theorem, we can assume that the above considered measurable functions are continuous and hence bounded on $\Gamma_H$ and consequently in each Lyapunov chart the deformation of a neighborhood of a point $x\in \Gamma_H$ under $f$ is approximately given by its linearization $df_x$ which, up to some tempered sequences of coordinate changes, is represented by diagonal matrices.  

This sketches the proof of Lemma~\ref{lllemma} under the assumption (A1).
\end{proof}

\begin{proof}[Sketch of proof of Lemma~\ref{lllemma} assuming (A2)]
Let us now assume that $f$ is a $C^1$ diffeomorphism preserving a hyperbolic measure $\mu$ such that the tangent space of $\Lambda=\supp \mu$ decomposes into a dominated splitting $T_\Lambda M=E^1\oplus E^2$, that is,  there exists $N\ge1$ such that for every $x\in \Lambda$, $v\in E^1_x$, $w\in E^2_x$ we have
\begin{equation}\label{eq:domi}
	\lVert df^N_x(v)\rVert \le \frac 12 \,\lVert df^N_x(w)\rVert\,.
\end{equation}
Up to a smooth change of metric on $M$ we can assume that $N=1$ (see~\cite[Appendix B]{BonDiaVia:05} and~\cite{Gou:07} for more details).

Observe that the dominated splitting is continuous and that it coincides with the splitting of the Oseledec decomposition at almost all points. This enables us to obtain the following non-uniform hyperbolicity on a large set without invoking Pesin theory. \cite[Lemma 8.4]{AbdBonCro:11} applied to $f$, $E^1$ and $f^{-1}$, $E^2$ simultaneously gives the following result.

\begin{lemma}\label{lem:subsubss}
	Given $\gamma\in(0,\chi)$, there exist $N=N(\gamma)\ge1$ 
	such that for $\mu$-almost every $x$ we have
	\begin{equation}\label{talcois}\begin{split}
		\lim_{\ell\to\infty}\frac{1}{\ell N}\sum_{k=0}^{\ell-1}
			\log\,\lVert df^N_{|E^1_{f^{kN}(x)}}\rVert
		&\le -\chi+\gamma\\
		\lim_{\ell\to\infty}\frac{1}{\ell N}\sum_{k=0}^{\ell-1}
			\log\, \lVert df^{-N}_{|E^2_{f^{-kN}(x)}}\rVert
		&\le -\chi+\gamma.
	\end{split}\end{equation}
\end{lemma}	

A point $x$ satisfying~\eqref{talcois} is called \emph{$N$-hyperbolic} with respect to $E^1\oplus E^2$. Up to a smooth change of metric on $M$, we can, and will, assume that $N=1$.

We now follow an idea in~\cite{BurWil:10} to construct  ``fake'' foliations   to foliate a uniformly sized \emph{neighborhood} of almost every point (they are called fake because they are close to but, in general, not identical with local stable/unstable manifolds).
Observe first that a dominated splitting can be extended to a dominated splitting to some neighborhood and that a dominated splitting persists under $C^1$-perturbations (see~\cite[Appendix B]{BonDiaVia:05}). 
The fake invariant foliations are constructed in two steps: (1) we find foliations tangent to a cone field in each tangent space $T_xM$ and (2)  using the exponential map, we project these foliations from some sufficiently small neighborhood of the origin in $T_xM$ to a neighborhood of $x$ in $M$. 

As above, we continue to consider $\rho_0>0$ so that for each $x\in M$ the exponential map $\exp_x\colon T_xM(2\rho_0)\to M$ is an embedding satisfying $\lVert(d\exp_x)_v\rVert\le2$ and so that $\exp_{f(x)}$ is injective on $\exp_{f(x)}^{-1}\circ f\circ \exp_x(T_xM(2\rho_0))$. 
Given $\rho\in(0,\rho_0)$, define the map $F_\rho\colon TM\to TM$ by 
\begin{itemize}
\item $F_{\rho,x}(v)=\exp_{f(x)}^{-1}\circ f\circ \exp_x(v)$ for $\lVert v\rVert\le \rho_0$,\\[-0.3cm]
\item $F_{\rho,x}(v)=df_x(v)$ for $\lVert v\rVert\ge 2\rho$, and\\[-0.3cm]
\item $\lVert F_{\rho,x}(\cdot)-df_x(\cdot)\rVert_{C^1}\to0$ as $\rho\to0$, uniformly in $x$.
\end{itemize}

Regarding the tangent bundle $TM$ as the disjoint union of its fibers, the tangent map $df$ has a dominated splitting. Even though the bundle $TM$ is not compact, all  relevant estimates hold uniformly for $F_{\rho,x}$ which is enough to conclude the claimed properties. 
As this splitting persists after a small $C^1$-perturbation, for sufficiently small $\rho$ the map $F_{\rho,x}$ has also a dominated splitting that we will denote by $\widetilde E^1\oplus \widetilde E^2$. Moreover, if $\rho$ is small enough, each subspace $\widetilde E^a$ at $v\in T_xM$ lies within the $\varepsilon/2$-cone about the corresponding subspace $E^a$ for $a=1,2$ (making the usual identification of $T_vT_xM$ with $T_xM$). 

As almost every $x$ is $1$-hyperbolic for $E^1\oplus E^2$ (with respect to $f$ on $\Lambda$) and as the splitting $\widetilde E^1\oplus \widetilde E^2$ is continuous and hence \eqref{talcois} varies continuously in $x$, we can assume that for sufficiently small chosen $\rho$ we obtain that every $v\in T_xM$ is $1$-hyperbolic with respect to $\widetilde E^1\oplus \widetilde E^2$ (with respect to $F_{\rho,x}$ on $TM$).
One can show (see, for example,~\cite[Section 8]{AbdBonCro:11}) that there exist $\rho_0>0$ such that in $T_xM$ the ball around the origin of radius $\rho_0$ is foliated  by leafs $\widetilde\cW^a_x(\cdot)$, $a=1,2$, having the following properties. For each $v$ the leaf $\widetilde\cW^a_x(v)$ is an injectively immersed $C^1$ submanifold tangent to $\widetilde E^1_x$ containing a ball of uniform size and being invariant in the sense that 
\[
	F_{\rho,x}(\widetilde\cW^1_x(v))\subset 
		\widetilde\cW^1_{f(x)}(F_{\rho,x}(v)),\quad
	F_{\rho,x}^{-1}(\widetilde\cW^1_x(v))\subset 
		\widetilde\cW^1_{f^{-1}(x)}(F_{\rho,x}^{-1}(v)).
\]
Having constructed these foliations $\widetilde\cW^a$ of $T_xM$, we apply the exponential map $\exp_x$ to it and obtain a foliation of the ball $B(x,\rho_0)$ that we will denote by $\cW^a(x)$,  $a=1,2$. If $\rho$ was chosen sufficiently small,  in a point $y\in B(x,\rho_0)$ the foliation will be within the $\varepsilon/2$-cone about the parallel translate to $E^a_x$. As the dominated splitting is transversal in all points, the leafs of the fake foliation intersect transversally and we can now define an almost rectangular box about $x$ that we  will denote by $R(x)$. 
By $1$-hyperbolicity of the point $x$ the leafs $\cW^1_x(\cdot)$ and $\cW^2_x(\cdot)$ will be contracted forward and backward in time under the map $f$ and $f^{-1}$, respectively.

Finally, the sequence~\eqref{def:psiss} (for $E^u=E^2$) converges uniformly on any large measure set $\Gamma_H$. Moreover, given $x\in \Gamma_H$, by continuity of the dominated splitting for any point $y$ whose orbit stays close to the orbit of $x$ the derivative in the direction $E^2$ differs only little from the one the one at $x$, proving the claim about the geometric potential.

This sketches the proof of  Lemma~\ref{lllemma} under the assumption (A2).
\end{proof}

\subsection{Uniform Birkhoff averages}\label{s:2.5}

Let $\phi\colon M\to\bR$ be a continuous function. The following lemma is an immediate consequence of the Birkhoff ergodic theorem and the Egorov theorem.

Given $n\ge1$ we define $S_n\phi\eqdef \phi+\phi\circ f+\ldots+\phi\circ f^{n-1}$.

\begin{lemma}\label{lem:biiir}
	Given  $\delta\in(0,1)$ and $r>0$, there exist a compact set $\Gamma_B=\Gamma_B(\delta,r,\phi)$ and a positive integer $n_B=n_B(\delta,r,\phi)$ such that $\mu(\Gamma_B)>1-\delta$ and for every $x\in \Gamma_B$ and every $n\ge n_B$ we have
\[
	\Big\lvert \frac{1}{n}S_n\phi(x) - \int \phi\,d\mu\Big\rvert \le r\,.
\]
\end{lemma}

\subsection{Uniform recurrence}\label{s:2.6}
Let $\cP=\{P_1,\ldots,P_N\}$ be any finite measurable partition of $M$. Given $x\in M$, denote by $\cP(x)$ the partition element which contains $x$. 

\begin{lemma}\label{lem:bur}
	Given numbers $\delta\in(0,1),r>0$, and a measurable set $A$ of positive measure, there exist a compact set $\Gamma_R=\Gamma_R(A,\delta,\cP,r)$ with $\mu(\Gamma_R)> 1-\delta$ and  a number $n_R=n_R(A,\delta,\cP,r)\ge 1$ such that for every $x\in \Gamma_R$ and every $n\ge n_R$ we have
\[
	f^k(x)\in\cP(x)\cap A\text{  for some number }k\in \{n,\ldots, n+rn\}\,.
\]	
\end{lemma}

\begin{proof}
Let $r\in(0,1)$ and 
\begin{equation}\label{defCCC}
	C\eqdef\min\Big\{r,\frac{\mu(P_i)}{4}\colon i=1,\ldots,N\Big\}.
\end{equation}
By the Birkhoff ergodic theorem, there is a $\mu$-full measure set $\widetilde A=\widetilde A(A,\cP)\subset M$ such that for every $i$ for which $\mu(\widetilde A\cap P_i)$ is positive and for every point in $\widetilde A\cap P_i$ we have 
\[
	\lim_{n\to\infty}\frac{1}{n}{\rm card}\big\{k\in\{0,\ldots,n-1\}
		\colon f^k(x)\in A\cap P_i\big\}
	=\mu(A\cap P_i)
	\,.
\]	 
By the Egorov theorem, we can conclude that for every such $i$ there exists a measurable set $\Gamma_i=\Gamma_i(A,\delta,\cP)\subset\widetilde A\cap P_i$ such that
\[
	\mu(\Gamma_i) \ge\mu(\widetilde A\cap P_i) - \frac\delta N
\]	 
and that convergence is uniform on $\Gamma_i$. Since $\mu$ is regular, without loss of generality, we can assume that each $\Gamma_i$ is compact.
Hence, we find $n_R=n_R(A,\delta,\cP,r)\ge 1$ such that for every such $i$, every $x\in\Gamma_i$, and every $n\ge n_R$ we have
	\[
	\left\lvert 
	\frac1n{\rm card}\big\{k\in\{0,\ldots,n-1\}\colon f ^k(x)\in A\cap P_i\big\} -
	 \mu(A\cap P_i)\right\rvert \le C^2.
	\]
We can assume that $n_R$ has been chosen large enough such that  for every such $i$ we have $n_Rr(\mu(A\cap P_i)-3C)>1$ and thus for every $x\in\Gamma_i$ and every $n\ge n_R$ we obtain
\[\begin{split}
	\card&\{k\in\{n,\ldots,n(1+r)-1\}\colon f ^k(x)\in A\cap P_i\} \\
	&\ge
 	n(1+r)\big(\mu(A\cap P_i)-C^2\big)
	-(n-1)\mu(A\cap P_i)-(n-1)C^2\\
	&=
 	(nr+1)\mu(A\cap P_i)-(2n+nr-1)C^2\\
	&\ge nr(\mu(A\cap P_i)-3C) >1\,,
\end{split}\]	
where we also used~\eqref{defCCC}. Let $\Gamma_R\eqdef \Gamma_1\cup\ldots\cup\Gamma_N$. 
By the above, we have
\[
	\mu(\Gamma_R) = \sum_{i=1}^N\mu(\Gamma_i)
	\ge \sum_{i=1}^N\left(\mu(\widetilde A\cap P_i) -\frac\delta N\right) 
	=\sum_{i=1}^N\mu(\widetilde A)-\delta
	=1-\delta.
\]
This proves the lemma.
\end{proof}

\subsection{Uniform growth of distinguishable orbits via entropy}

To construct a set on which dynamics approximates entropy, the following lemma enable us to produce a sufficiently large number of points which have distinguishable orbits of a certain length. 

Given $\varepsilon>0$, $n\ge1$, and a compact set $K\subset M$, we call a set $E\subset M$ \emph{$(n,\varepsilon)$-spanning} for $K$ (with respect to $f$) if 
\[
	K\subset \bigcup_{x\in E}B_n(x,\varepsilon)\,,
\]
where
\[
B_n(x,\varepsilon)\eqdef 
\big\{y\colon d(f ^k(x),f ^k(y))\le \varepsilon\text{ for every }0\le k\le n-1\big\}\,.
\] 
We call a set $E\subset M$ \emph{$(n,\varepsilon)$-separated} (with respect to $f$) if for every $x\in E$ the set $B_n(x,\varepsilon)$ contains no other point of $E$.

\begin{lemma}\label{l:recurrrr}
	 Given $\delta\in(0,1)$, for every $e\in(0,h_\mu(f)]$, $r\in(0,e)$, and $\varepsilon>0$ there exist a number $n_E=n_E(\delta,e,r,\varepsilon)$ and a measurable set $\Gamma_E$ satisfying $\mu(\Gamma_E)>1-\delta$ such that for every $n\ge n_E$ and for every measurable set $A\subset \Gamma_E$ with $\mu(A)>0$ there exists a $(n,\varepsilon)$-separated set $E\subset A$ such that
\[
 	- \frac1n\lvert\log\mu(A)\rvert - r\le
	\frac1n \log\card E
	- e
	\le  r \,.
\]
\end{lemma}

\begin{proof}
By ergodicity, by the Brin-Katok theorem~\cite{BriKat:83}, for $\mu$-almost every $x$
\[
	\lim_{\varepsilon\to0}\limsup_{n\to\infty}-\frac1n\log\mu(B_n(x,\varepsilon))
	=h_\mu(f)\,.
\]
Given $r$ and $\varepsilon$, by the Egorov theorem there is a set $\Gamma_E$ of measure $\mu(\Gamma_E)>1-\delta$ and a number $n_E\ge1$ such that for every $n\ge n_E$ and every $x\in \Gamma_E$ 
\[
	e^{-n(h_\mu(f)+r)}\le\mu(B_n(x,\varepsilon))\le e^{-n(h_\mu(f)-r)}\,.
\]
For any positive measure set $A\subset \Gamma_E$, we construct an $(n,\varepsilon)$-separated set as follows. Choose any point $x_1\in A$. Let $A_1\eqdef A\setminus B_n(x_1,\varepsilon)$. Continuing inductively, for every $k\ge 2$ choose any $x_k\in A_{k-1}$, let $A_k\eqdef A_{k-1}\setminus B_n(x_k,\varepsilon)$. Since we have 
\[
	\mu(A_k)\ge \mu(A)-k\mu(B_n(x_k,\varepsilon))
	\ge   \mu(A)-ke^{-n(h_\mu(f)-r)},
\]
we can repeat this iteration for at least $\lceil\mu(A)\cdot e^{n(h_\mu(f)-r)}\rceil$ times. The resulting set of points $E\eqdef\{x_1,x_2,\ldots\}\subset A$ is by construction $(n,\varepsilon)$-separated and  satisfies
\[
	\frac1n\log\card E \ge  h_\mu(f)- r - \frac1n\lvert\log\mu(A)\rvert\,.
\]
Possibly after neglecting some elements of $E$, we can guarantee that 
\[
	h_\mu(f)+r\ge e+r\ge\frac1n\log\card E 
	\ge e - r - \frac1n\lvert\log\mu(A)\rvert\,,
\]
which proves the lemma.
\end{proof}

\section{Proofs}\label{s:3}

We now provide the proofs of the theorems. Applying Sections~\ref{s:2.4},~\ref{s:2.5}, and~\ref{s:2.6}, we largely follow~\cite[Supplement S.5]{KatHas:95}. We also apply ideas from~\cite[Section 11.6]{PrzUrb:10}.

\begin{proof}[Proof of Theorem~\ref{t:1}]
	Let $U\subset M$ be an open and bounded neighborhood of $\supp\mu$. Pick a countable basis $\{\psi_i\}_{i\ge1}$ of continuous (nonzero) functions in the space $C^0(\overline U)$ of all continuous functions on $\overline U$. Recall that the space of invariant probabilities on $\overline U$ the following  function  $d\colon \cM\times\cM\to[0,1]$
	\begin{equation}\label{distance}
		d(\mu,\nu)\eqdef \sum_{j=1}^\infty 
		2^{-j}\frac{1}{2\lVert\psi_j\rVert_\infty}
		\,\Big\lvert \int\psi_j\,d\mu-\int\psi_j\,d\nu\Big\rvert\,,
	\end{equation}	
where $\lVert \psi\rVert_\infty\eqdef\sup_{x\in\overline U}\,\lvert\psi(x)\rvert$, provides a metric which induces the weak$\ast$ topology on $\cM$.	
	
Fix some 
\[
	\delta\in(0,1/5)
\]
and let
\[	
	r\in(0,\min\{h_\mu(f),\chi/3\})\,.
\]
Let $K$ be a positive integer satisfying 
\begin{equation}\label{choicer00}
	2^{-K+1}< \frac r2
\end{equation}	
and choose $r_0\in(0,r)$ such that
\begin{equation}\label{choicer0}
	r_0(1-2^{-K})\frac{1}{2}\max_{j=1,\ldots,K}\lVert\psi_j\rVert_\infty^{-1} < \frac r 2\,.
\end{equation}
Moreover, assume that $K$ was chosen large enough such that $\{\psi_1,\ldots,\psi_K\}$ is $r$-dense, that is, for every $\varphi\in C^0(\overline U)$ there exists $\psi_j, j\in\{1,\ldots,K\},$ such that 
\begin{equation}\label{e:labsss}
	\lVert \psi_j-\varphi\rVert_\infty<r\,.
\end{equation}
Choose $\varepsilon_0>0$ sufficiently small such that  the modulus of continuity of each $\phi\in\{\psi_1,\ldots,\psi_K\}$ satisfies 
\begin{equation}\label{e:modcon}
	\sup_{x,y,d(x,y)\le\varepsilon_0}\lvert \phi(x)-\phi(y)\rvert< r_0\,.
\end{equation}	
For the following arguments, for simplicity, we will consider only one potential function $\phi$ and will assume that they hold simultaneously for every potential $\psi_j$, $j=1,\ldots,K$.

First, by Lemma~\ref{lem:mulergthe}, there is a compact set $\Gamma_J=\Gamma_J(\delta,r)$ and a positive integer $n_J=n_J(\delta,r)$ such that $\mu(\Gamma_J)>1-\delta$ and for every $x\in \Gamma_J$ and every $n\ge n_J$ we have
\begin{equation}\label{e:Jacobian}
	\Big\lvert 
	\frac1n \log\,\lvert \Jac df^n_{|E^u_x}\rvert
	 - \sum_j\chi_j(\mu)^+
	 \Big\rvert \le r\,.
\end{equation}	

Next, we will choose points in some Lyapunov regular set. Hence they will have  hyperbolic behavior and a nice rectangular cover. 
Therefor fix 
\[
	\gamma\in(0,\chi/3),\quad\varepsilon_1\in(0,\varepsilon_0)\,.
\]	
By Lemma~\ref{lllemma}, there exist a compact set $\Gamma_H=\Gamma_H(\gamma,\varepsilon_1,\delta)$ of measure $> 1-\delta$ and positive numbers $\lambda\in(0,1),\rho,L$ and a finite collection of rectangles $R(x_1), \ldots, R(x_i),\ldots,R(x_\ell)$ with $x_i\in \Gamma_H$ and so that $B(x_i,\rho)\subset R(x_i)$ and $\diam R(x_i)<\varepsilon_1$ for every $i$, satisfying properties 0.--5. in the lemma.

By Lemma~\ref{lem:biiir}, there is a compact set $\Gamma_B=\Gamma_B(\delta,r,\phi)$ and a positive integer $n_B=n_B(\delta,r,\phi)$ such that $\mu(\Gamma_B)>1-\delta$ and for every $x\in \Gamma_B$ and every $n\ge n_B$ we  have
\begin{equation}\label{e:Birkhofff}
	\left\lvert \frac{1}{n}
		S_n\phi(x) - \int \phi\,d\mu\right\rvert \le r\,.
\end{equation}	

Besides the rectangles, let us fix a finite partition $\cP$ of $M$ of diameter $<\rho/2$. Notice that in this way for each $x_i$ the partition element $\cP(x_i)$ is completely contained in $R(x_i)$. 

Further, we want points in $\Gamma_J\cap\Gamma_H\cap \Gamma_B$, or at least most of them, to be closely recurring at the same, or at least at almost the same, time. Note that $\mu(\Gamma_J\cap\Gamma_H\cap\Gamma_B)>1-3\delta>0$.
By Lemma~\ref{lem:bur}, there is a compact set $\Gamma_R=\Gamma_R( \Gamma_J\cap\Gamma_H\cap \Gamma_B,\delta,\cP,r)$ of points with almost uniform recurrence times  satisfying $\mu(\Gamma_R)>1-\delta$ 
and a number $n_R=n_R( \Gamma_J\cap\Gamma_H\cap \Gamma_B,\delta,\cP,r)$ such that for every $n\ge n_R$ and every $x\in \Gamma_R$ for some $k\in\{n,\ldots,n+nr\}$ we have
\begin{equation}\label{rec}
	f^k(x)\in\cP(x)\cap \Gamma_J\cap\Gamma_H\cap \Gamma_B\, .
\end{equation} 
Consider the set 
\[
	\Gamma'\eqdef \Gamma_J\cap \Gamma_H\cap \Gamma_B\cap \Gamma_R,
\] 
and note that it  satisfies  $\mu(\Gamma')> 1-4\delta>0$. 

Further, by Lemma~\ref{l:recurrrr}, given $e\in(0,h_\mu(f)]$, for every $\varepsilon\in(0,\varepsilon_1)$ there exist a number $n_E=n_E(\delta,e,r,\varepsilon)$ and a set $\Gamma_E$ satisfying $\mu(\Gamma_E)> 1-\delta$ such that for every 
\begin{equation}\label{nchoi}
	n\ge
	\max\left\{n_J,n_H,n_B,n_R,n_E,\frac{\log\ell}{r},
			\frac{\lvert\log(1-5\delta)\rvert}{r}\right\}
\end{equation} 
 there is a $(n,\varepsilon)$-separated set $E\subset \Gamma'\cap \Gamma_E$ such that 
\begin{equation}\label{spaet-E}
	(1-5\delta)e^{n(e-r)}\le\card E\le e^{n(e+r)}\,,
  \end{equation}  
where we use the fact that $\mu(\Gamma'\cap \Gamma_E)\ge 1-5\delta>0$.  Let $\Gamma''\eqdef \Gamma'\cap \Gamma_E$.
  
For each $k$ with $n\le k< n+r n$ let
\[
E_k\eqdef \left\{x\in E\colon f ^{k}(x)\in\cP(x) \right\} 
\]
be the set of points in $E$ that have the same time $k$ of return to their
partition element.  Pick an index $m$  satisfying
\[
	\card E_m=\max_{n\le k<n+rn}\card E_k\,.
\]
Since $\card  E= \sum_{n\le k< n+r n}\card  E_k$, we have $rn\, {\rm  card}\,E_m\ge\card E$. With $rn <e^{nr}$ and~\eqref{spaet-E} we obtain
\[
	(1-5\delta)e^{n(e-2r)}
	\le \frac{\card E}{rn}\le \card E_m \le\card E\le  
	 e^{n(e+r)}.
\]
Pick now out of the $\ell$ rectangles from the rectangular cover the element $R(x_i)$, $x_i\in \Gamma_H$, for which $\card (E_m\cap \cP(x_i))$ is maximal. Hence we have 
\begin{equation}\label{docher}
	\frac{1}{\ell}(1-5\delta) e^{n(e-2r)}
	\le \card  (E_m\cap \cP(x_i))
	\le \frac{1 }{\ell} \card  E_m
	\le  \frac{1}{\ell} e^{n(e+r)}\,.
\end{equation}

Recall that exactly after $m$ iterations each point $x\in E_m\cap \cP(x_i)$ returns to $\cP(x_i)$, and hence to the rectangle $R(x_i)$, and that $f ^m(x)\in \Gamma_J\cap\Gamma_H\cap\Gamma_B$ (recall~\eqref{rec}). We have $x\in B(x_i,\rho/2)$ and $f^m(x)\in B(x_i,\rho/2)$. 
By items 3.--4. in Lemma~\ref{lllemma}, the connected components 
\[
\cC_x\left(R(x_i)\cap f^{-m}R(x_i)\right) 
\quad\text{and}\quad
f^m\left(\cC_x\left(R(x_i)\cap f^{-m}R(x_i)\right)\right)
\]
are an admissible $L^s$-rectangle and  $L^u$-rectangle in $R(x_i)$, respectively, for some number $L>0$. Analogously, one can show that the connected components 
\[
\cC_{f^m(x)}(R(x_i)\cap f^mR(x_i)) 
\quad\text{and}\quad
f^{-m}(\cC_{f^m(x)}(R(x_i)\cap f^mR(x_i)))
\]
are an admissible $L^u$-rectangle and  $L^s$-rectangle in $R(x_i)$, respectively. 
Note that item 4. in Lemma~\ref{lllemma} implies that any point $y\in \cC_x(R(x_i)\cap f^{-m}R(x_i))$ with $y\ne x$ satisfies 
\[
d(f^k(x),f^k(y))\le 3\diam R(x_i)\cdot
	\max\{e^{-k(\chi-\gamma)},e^{-(m-k)(\chi-\gamma)}\}
	<\varepsilon
\] 
for each $k=0, \ldots, m$, which implies that $y\notin  E_m\cap \cP(x_i)$ since $E_m$ is an $(n,\varepsilon)$-separated set and $x\in E_m\cap\cP(x_i)$. This implies that there are $\card  (E_m\cap \cP(x_i))$ disjoint admissible $L^s$-rectangles which are mapped under $f^m$ onto $\card  (E_m\cap \cP(x_i))$ admissible $L^u$-rectangles.
Let 
\[
\Gamma\eqdef 
\bigcap_{n\in\bZ}f^{nm}\left(\bigcup_{x\in  E_m\cap \cP(x_i)}
\cC_x\left(R(x_i)\cap f^{-m}R(x_i)\right) \right)\,.
\]
Consider the set $V=R(x_1)\cup\ldots\cup R(x_\ell)$ and the map $g=f^m$.
Observe that by construction we have
\[
	\Gamma\eqdef\bigcap_{k\in\bZ}g^k(V)\,,
\]
that is, $\Gamma$ is locally maximal with respect to $g$ and to the closed neighborhood $V$ built by rectangles.
By construction, $f^m|_\Gamma$ is topologically conjugate to a full two-sided shift in the symbolic space with $\card  (E_m\cap \cP(x_i))$ symbols. This shows claim i) of the theorem.

Since the rectangular cover was built around some sufficiently small neighborhood of points in the set $\Gamma_H$ which is a subset of the support of the measure $\mu$, claim ii) of the theorem follows.

Observe that every orbit in $\Gamma$ is modeled over the shift and the map $f^m$ sends a point in some rectangle $R(x_i)$ into a rectangle $R(x_j)$ by an orbit that stays $\varepsilon$-close to an orbit satisfying~\eqref{e:Birkhofff} for $n=m$. Consider the set
\[
	\widehat\Gamma\eqdef \Gamma\cup f(\Gamma)\cup\ldots\cup f^{m-1}(\Gamma)
\]
and observe that it is $f$-invariant.
Together with~\eqref{e:modcon} we obtain that for every $x\in\widehat\Gamma$ we have
\[
	\int\phi\,d\mu - r\le
	\liminf_{n\to\infty}\frac{1}{nm}S_{nm}\phi(x)\le
	\limsup_{n\to\infty}\frac{1}{nm}S_{nm}\phi(x)\le\int\phi\,d\mu + r\,.
\]
In particular, for every $f$-invariant probability measure $\nu$  supported on $\widehat\Gamma$ 
\begin{equation}\label{botafogo}
	\Big\lvert\int\phi\,d\nu-\int\phi\,d\mu\Big\rvert < r\,.
\end{equation}
Now recall that we concluded the above for all $\phi\in\{\psi_1,\ldots,\psi_K\}$. Hence, with~\eqref{distance},~\eqref{choicer00}, and~\eqref{choicer0} we obtain 
\[\begin{split}
	d(\nu,\mu)
	&\le 
	\sum_{j=1}^K2^{-j}\frac{r_0}{2\lVert\psi_j\rVert_\infty}
	+\sum_{j=K+1}^\infty 
		2^{-j}\frac{1}{2\lVert\psi_j\rVert_\infty}
		\,\Big\lvert \int\psi_j\,d\mu-\int\psi_j\,d\nu\Big\rvert\\
	&\le (1-2^{-K})\frac{r_0}{2}\max_{j=1,\ldots,K}\lVert\psi_j\rVert_\infty^{-1} 
		+ 2^{-K}2
	< r\,.
\end{split}\]
This proves claim iii) of the theorem.

Since $f^m|_\Gamma$ is conjugate to the full shift, its topological entropy is equal to $\log\card  (E_m\cap \cP(x_i))$ and hence the entropy of $f|_\Gamma$ satisfies
\[
h(f|_\Gamma)=\frac{1}{m}\log\,\card  (E_m\cap \cP(x_i))\,.
\]
Together with~\eqref{docher} we obtain
\[
\frac{1}{m}\log\,\frac{1}{\ell} -\frac1m\lvert\log(1-5\delta)\rvert
+ \frac{n}{m}(e-2r)
\le h(f|_\Gamma)
\le \frac{1}{m}\log\,\frac{1}{\ell} + \frac{n}{m}(e+r)\,.
\]
Recall that $m< n+rn$ and hence $n/m> 1-r/(1+r)$ and together with~\eqref{nchoi} 
\begin{equation}\label{eq:heutei}\begin{split}
h(f|_\Gamma)
&\ge e - 2r -  \frac{r}{1+r}(e-2r)-\frac1m\lvert\log(1-5\delta)\rvert- \frac{1}{m}\log\,\ell\\
&>  e -  r(2+e) -2r.
\end{split}\end{equation}
On the other hand
\begin{equation}\label{eq:heuteib}
	h(f|_\Gamma) \le \frac1m\log\frac1\ell +e+r<e+r\,.
\end{equation}
This proves that $h(f|_\Gamma)$ is arbitrarily close to $e$,  provided that $r$ has been chosen small enough.
This implies claim iv) of the theorem.

Let $\nu\in\cM(f|_{\widehat\Gamma})$ and $\varphi\colon M\to\bR$ continuous. Since we want to estimate $\int\varphi\,d\nu$, it suffices to assume $\varphi\in C^0(\overline U)$.
Then, by~\eqref{e:labsss} there exists some $\psi_i$ such that $\lVert \psi_i-\varphi\rVert_\infty<r$. Hence, together with~\eqref{botafogo}
\begin{equation}\label{humaita}\begin{split}
	\int\varphi\,d\nu &\ge \int\psi_i\,d\nu - r
	\ge \int\psi_i\,d\mu-2r
	\ge \int\varphi\,d\mu-3r\,.
\end{split}\end{equation}
Recall that the \emph{topological pressure of $\varphi$}  with respect to the compact $f$-invariant set $\widehat\Gamma$ 
satisfies the \emph{variational principle} (see~\cite{Wal:81})
\begin{equation}\label{eq:vprince}
	P(\varphi,f|_{\widehat\Gamma}) 
	= \sup_{\nu\in\cM(f|_{\widehat\Gamma})}\big(h_\nu(f)+\int\varphi\,d\nu\big)\,.	
\end{equation}
Thus,  together with~\eqref{humaita} and using~\eqref{eq:heutei}, we obtain
\[\begin{split}
	P(\varphi,f|_{\widehat\Gamma})
	&=\sup_{\nu\in\cM(f|_{\widehat\Gamma})}\big(h_\nu(f)+\int\varphi\,d\nu\big)
	\ge \sup_{\nu\in\cM(f|_{\widehat\Gamma})} h_\nu(f)+\int\varphi\,d\mu - 3r\\
	&= h(f|_{\widehat\Gamma})+\int\varphi\,d\mu - 3r
	\ge e +\int\varphi\,d\mu - r(7 + e) \,.
\end{split}\]
Analogously, using~\eqref{eq:heuteib} instead, we obtain
\[\begin{split}
	P(\varphi,f|_{\widehat\Gamma})
	&=\sup_{\nu\in\cM(f|_{\widehat\Gamma})}\big(h_\nu(f)+\int\varphi\,d\nu\big)
	\le h(f|_{\widehat\Gamma}) +\int\varphi\,d\mu + 3r\\
	&< e + \int\varphi\,d\mu + 4r\,.
\end{split}\]
This implies claim v) of the theorem.

Finally, by item 5. in Lemma~\ref{lllemma} and~\eqref{e:Jacobian} we conclude vi).
\end{proof}

\begin{proof}[Proof of Theorem~\ref{t:2}]
Let $e\in[0,h(f|_M))$.
By the variational principle, there exists an ergodic measure $\mu$ with $h_\mu(f)\in(e,h(f|_M)]$. Hence, by Theorem~\ref{t:1}, there exists $m\ge1$ and a basic set $\Gamma\subset M$ (with respect to $f^m$) satisfying $h(f|_\Gamma)=e+\zeta$ for some positive number $\zeta\in (0,h(f|_M)-e)$. The map $f^m|_\Gamma$ is topologically conjugate to a full two-sided shift. Then (see, for example~\cite[p. 178--179]{Wal:81})  for every $\gamma\in(0,e+\zeta)$ there exists a closed shift-invariant subset with topological entropy $\gamma$.
The claim is now an immediate consequence.
\end{proof}

\bibliographystyle{amsplain}

\end{document}